\documentclass[11pt]{article}

\usepackage{amsmath,amsfonts,amssymb,thc,dsfont,graphicx,a4wide,enumerate,tikz}

\theoremstyle{change}
\theorembodyfont{\itshape}
\newtheorem{theorem}{\bf Theorem}[section]
\newtheorem{proposition}[theorem]{\bf Proposition}
\newtheorem{lemma}[theorem]{\bf Lemma}
\newtheorem{corollary}[theorem]{\bf Corollary}

\theorembodyfont{\rm}
\newtheorem{definition}[theorem]{\bf Definition}
\newtheorem{remark}[theorem]{\bf Remark}

\newtheorem{example}[theorem]{\bf Example}

\newcommand{\conver}{\mathop{\longrightarrow}}
\newcommand{\intau}{\ \conver_{\tau}\ }
\newcommand{\ines}{\ \conver_{S}\ }
\newcommand{\intauw}{\ \conver_{\tau_w}\ }
\newcommand{\intaues}{\ \conver_{\tau_s}\ }

\newcommand{\notintaues}{\ \not\hspace{-3mm}{\conver_{\tau_s}}\ }
\newcommand{\notintau}{\ \not\hspace{-3mm}{\conver_{\tau}}\ }

\newcommand{\cL}{{\mathcal L}}

\newcommand{\cX}{{\mathcal X}}

\newcommand{\cC}{{\mathcal C}}

\newcommand{\cK}{{\mathcal K}}

\newcommand{\GN}{{\mathds{N}}}
\newcommand{\GR}{{\mathds{R}}}

\newcommand{\GI}{{\mathds{I}}}

\newcommand{\GH}{{\mathds{H}}}

\newcommand{\GD}{\mathds{D}}
\newcommand{\GV}{\mathds{V}}

\newcommand{\bP}{{\mathbb{P}}}

\newcommand{\bV}{\mathbb{V}}

\newcommand{\bproof}{\noindent{\sc Proof.}}
\newcommand{\eproof}{$\Box$}

\begin{document}
	\title{C{\`a}dl{\`a}guity and compactness in submetric spaces}

	\author{Adam Jakubowski\footnote{E-mail: adjakubo@mat.umk.pl}  \\[3mm]
		Nicolaus Copernicus University, Toru\'n, Poland
	}

	\date{}
	
	\maketitle
	
	\begin{abstract}
		A submetric space is a topological space equipped with a continuous metric that generates a metric topology weaker than the original one (e.g., a separable Hilbert space with the weak topology). In this paper, we systematically investigate the compactness-related properties of submetric spaces, which make this class of topological spaces a particularly suitable tool for probabilistic applications. As an illustration, we discuss in detail the compactness of the closure of a c\`adl\`ag trajectory taking values in a submetric space.
	\end{abstract}

	
	\noindent{\em MSClassification 2020:} 
	60F17; 60G17; 54D30; 54D55; 54B99.

\tikzset{every picture/.style={line width=0.75pt}}

\section{Introduction}

In the theory of stochastic processes with values in topological spaces a special place is occupied by the so-called compact containment condition (CCC). Let $\{X_t\}_{t\in[0,T]}$ be a stochastic process taking values in a topological space $(\cX,\tau)$. We say that $\{X_t\}$ satisfies the CCC, if for each $\varepsilon > 0$ one can find a \emph{$\tau$-compact} subset $K_{\varepsilon}$ of $\cX$ such that
\begin{equation}\label{eq:ccc}
	\bP\big( X_t \in K_{\varepsilon}, t\in [0,T]\big) > 1-\varepsilon.
\end{equation}  

If CCC holds uniformly for a family of stochastic processes, then it is often possible to reduce problems of uniform tightness of this family to the uniform tightness of real-valued stochastic processes. Theorem 3.1 in \cite{Jak86} may serve as a commonly used example, where processes with c\`adl\`ag (i.e., right-continuous and admitting left limits at all $t \in (0,T]$) trajectories are considered.

It is obvious that CCC may hold only if a single trajectory can be included in a compact set. Recently Janson \cite{Jan26} obtained a general result of this type. He proved that in a regular topological space $(\cX,\tau)$, for every c\`adl\`ag function $x: [0,T] \to  \cX$ the  set  $ K_x = \{ x(t)\, ; \, t\in [0,T]\} \cup \{ x(t-)\, ; \, t\in (0,T]\}$ is compact. Moreover, by providing a counterexample, he showed that this property is not valid in general Hausdorff spaces  

The purpose of this note is to discuss an analogous question in submetric spaces, a class of topological spaces particularly suitable for the needs of probability theory. 

In Section \ref{Sec:SS} we define submetric spaces and describe their position in the hierarchy of topological spaces. Section \ref{Sec:useful} presents several properties of submetric spaces that illustrate their usefulness in probability theory and stochastic processes; here, we also refer to the expository paper \cite{Jak23}. In Section \ref{Sec:relcomp}, we discuss the relationship  between c\`adl\`aguity and relative compactness in submetric spaces. Finally, Section \ref{Sec:Hausdorff} covers so-called reversibly c\`adl\`ag functions in general Hausdorff spaces.

The paper is intended for readers with a background in probability theory. Consequently, all topological results are either provided with detailed proofs or carefully referenced to the literature.

\section{Submetric spaces}\label{Sec:SS}
\begin{definition}
	If $\big(\cX,\tau\big)$ is a topological space and $d$ is a metric on $\cX$ such that the metric topology $\tau_d$ generated by $d$ is contained in $\tau$, then we shall say that the metric $d$ is compatible with $\tau$.
\end{definition}
It is not difficult to show that the inclusion $\tau_d \subset \tau$ holds if, and only if, for each $x_0 \in \cX$ the function $\cX \ni x \mapsto d(x,x_0)\in \GR^+$ is $\tau$-continuous.
\begin{definition}
	A topological space  $\big(\cX,\tau\big)$ is called {\em submetric}, if there exists a metric $d$ compatible with $\tau$. 
\end{definition}
It should be stressed that in general $\tau_d \subsetneq\tau$  i.e. $\tau_d$ is essentially weaker than $\tau$.

The name ``submetric" is taken from \cite{Gru84}, where  a topological characterization of submetric spaces - how to get rid of the metric in the definition - was also given.  
Our interest in submetric spaces is of different nature - we exhibit the useful consequences of the existence of a~continuous metric.

The class of submetric spaces is quite large. Beyond metric spaces it contains, for example, spaces $(\cX,\tau)$, on which there exists a sequence  $\{f_i\}_{i\in\GN}$ of $\tau$-continuous functions separating the points in $\cX$ (i.e. if $f_i(x) = f_i(y)$, $i\in\GN$, then $x=y$).  If such a family is given then  
\begin{equation}\label{eq:2} \cX \times \cX \ni (x,y) \mapsto d(x,y) = \sum_{i=1}^{\infty} \frac{1}{2^i} \frac{|f_i(x) - f_i(y)|}{1 + |f_i(x) - f_i(y)|}.
\end{equation}
is a {\em continuous} metric on $\cX$. It follows that many weak topologies on linear topological spaces are submetric. A separable Hilbert space equipped with the weak topology can serve here as the prime example (see Example \ref{Ex:Hilb} below).

The position of submetric spaces among other general classes of topological spaces may be described by the following diagram, where all inclusions are {\em strict}.
\begin{center}
	\begin{tikzpicture}[x=0.75pt,y=0.75pt,yscale=-1,xscale=1]

		\draw (70,125) node [anchor=north west][inner sep=0.75pt]   [align=left] {MS};
		\draw (98,125) node [anchor=north west][inner sep=0.75pt]    {$\subset $};
		\node[draw,circle,inner sep=1pt,line width=0.25pt] at (105,145)  {\footnotesize $1$};
		\draw (120,125) node [anchor=north west][inner sep=0.75pt]   [align=left] {SS};
		\draw (140,125) node [anchor=north west][inner sep=0.75pt]    {$\cap $};
		\draw (155,125) node [anchor=north west][inner sep=0.75pt]   [align=left] {CR};
		\node[draw,circle,inner sep=1pt,line width=0.25pt] at (190,107)  {\footnotesize $2$};
		
		\draw (180,122) node [anchor=north west][inner sep=0.75pt]  [rotate=30]  {$\subset $};
		\draw (185,130) node [anchor=north west][inner sep=0.75pt]  [rotate=-30]  {$\subset $};
		
		\node[draw,circle,inner sep=1pt,line width=0.25pt] at (190,155)  {\footnotesize $3$};
		\draw (200,112) node [anchor=north west][inner sep=0.75pt]   [align=left] {SS};
		\draw (200,138) node [anchor=north west][inner sep=0.75pt]   [align=left] {CR};
		\node[draw,circle,inner sep=1pt,line width=0.25pt] at (233,155)  {\footnotesize $5$};
		\draw (224,137) node [anchor=north west][inner sep=0.75pt]  [rotate=30]  {$\subset $};
		\draw (230,113) node [anchor=north west][inner sep=0.75pt]  [rotate=-30]  {$\subset $};
		
		\node[draw,circle,inner sep=1pt,line width=0.25pt] at (233,107)  {\footnotesize $4$};
		\draw (247,125) node [anchor=north west][inner sep=0.75pt]   [align=left] {HS};
	\end{tikzpicture}
\end{center}
Here MS means {\em metric spaces}, SS - {\em submetric spaces}, CR - {\em completely regular spaces}, and HS - {\em Hausdorff spaces}.
\begin{example}[Weak topology on a Hilbert space]\label{Ex:Hilb}
	Let $\GH$ be a separable infinite dimensional (real) Hilbert space with the inner product $\langle \cdot,\cdot\rangle$. By the weak topology $\tau_w$ on $\GH$ we mean the topology generated by all linear functionals $\{ \langle\cdot, y\rangle\,;\, y \in \GH\}$ (i.e. the coarsest topology $\tau$ such that all these functionals are $\tau$-continuous). The basis of neighborhoods of $0$ in $\tau_w$ is uncountable and so the weak topology is not metrizable (see e.g. \cite[Theorem 1, p.40]{Jar81}). On the other hand, let us take an orthonormal basis $\{e_i\}_{i\in\GN}$ in $\GH$  and define
	\[ \GH \times \GH \ni (x,y) \mapsto d(x,y) = \sum_{i=1}^{\infty} \frac{1}{2^i} \frac{|\langle x-y, e_i\rangle|}{1 + |\langle x-y, e_i\rangle|}.\]
	It is a continuous metric on $\GH$. Hence $(\GH,\tau_w)$ is a submetric space. Being a linear topological space it is also a completely regular space  \cite[p.16]{Sch71}. Recall that a topological space $(\cX,\tau)$ is {\em completely regular} if for each $\tau$-closed set $F$ and each point $x_0\not\in F$ there exists a continuous function $f : \cX \to [0,1]$ such that $f(x_0) = 1$ and $f(x) = 0$ if $x \in F$. We conclude that the inclusion no. 1 in the above diagram is strict.
\end{example}

\begin{example}[Smirnov space]\label{Ex:Smir}
	This example was constructed by Smirnov in \cite[Example 1, p. 107]{Smir51}. We shall follow the formalism of \cite{Eng89}. Let $\cX = \GR^1$ and let $Z$ be the set of reciprocals of all positive integers. Consider the system of neighborhoods of $x\in \GR^1$ defined by
	\begin{equation}
		B(x) = \begin{cases} \{ (x-1/i,x+1/i)\,;\, i\in \GN\},& \text{ if $x \neq 0$}, \\
			\{ (x-1/i,x+1/i)\setminus Z\,;\, i\in \GN\},& \text{ if $x = 0$.}
		\end{cases}
	\end{equation}
	Then it was proved in \cite[Example 1.5.6, p. 39]{Eng89} that $\{B(x)\}_{x\in\GR^1}$ is a basis for some topology $\tau$ (the ``Smirnov topology'') and the space $(\cX,\tau)$ is not a regular space (hence also not a completely regular topological space).
	
	On the other hand, any interval $(-\varepsilon,+\delta)$ contains some  set $(x-1/i,x+1/i)\setminus Z$ (but not conversely), hence $\tau$ is {\em strictly} stronger than the usual - metric - topology of $\GR^1$. It follows that the Smirnov space is a submetric space. This proves that the inclusion no.2 in the diagram is strict. 
	
\end{example}

\begin{example}[Tychonoff cube]\label{Ex:Tych}
	For our purposes it is enough to consider the Tychonoff cube of the weight of continuum, e.g. the Cartesian product $[0,1]^{[0,1]}$. It is a completely regular topological space by \cite[Theorem 2.3.11]{Eng89} and is not metrizable, as the basis at every $x \in [0,1]^{[0,1]}$ is uncountable. But this is not enough for our purposes. We want to show that the Tychonoff cube is not {\em submetrizable}. Assume there exists a continuous (with respect to the product topology $\tau_{\Pi}$) metric $d$ on $[0,1]^{[0,1]}$. Then the identity map $I : \big([0,1]^{[0,1]},\tau_{\Pi}\big) \to \big([0,1]^{[0,1]},\tau_d\big)$ is a continuous and one-to-one map of the compact space $\big([0,1]^{[0,1]},\tau_{\Pi}\big)$ onto the Hausdorff space $\big([0,1]^{[0,1]},\tau_d\big)$. By \cite[Theorem 3.1.13]{Eng89} the map $I$ is a homeomorphism, hence $\big([0,1]^{[0,1]},\tau_{\Pi}\big)$ is metrizable. A contradiction.  This proves that the inclusion no. 3 is strict.
\end{example}
The inclusions no. 4 and no. 5 are obvious. 

\section{Which compactness properties render submetric spaces useful in probability theory?}\label{Sec:useful}

Before presenting the results, let us first establish  the core terminology. A subset $K$ of a Hausdorff space $(\cX,\tau)$ is said to be:
\begin{description}
	\item{\bf compact}, if for arbitrary family $U_i,\ i\in \GI$, of $\tau$-open sets such that $K \subset \cup_{i\in\GI} U_i$ there exists a finite subset $\GI_0 \subset \GI$ such that $K \subset \cup_{i\in\GI_0} U_i$.
	\item{\bf sequentially compact}, if in every sequence $\{x_n\} \subset K$ one can find a subsequence $\{x_{n_k}\}$ which is $\tau$-convergent to a limit $x_0$ belonging to $K$: 
	\[x_{n_k} \intau x_0 \in K.\]
	\item{\bf relatively compact}, if in every sequence $\{x_n\} \subset K$ one can find a subsequence $\{x_{n_k}\}$ which is $\tau$-convergent (to a limit $x_0$ which need not belong to $K$):
	\[x_{n_k} \intau x_0 \in \cX.\]
	\item{\bf conditionally compact}, if its $\tau$-closure is $\tau$-compact.
\end{description}
It is well-known that in metric spaces the four types of compactness-like notions reduce to two: 
\begin{enumerate}
	\item compactness and sequential compactness coincide;
	\item relative compactness and  conditional compactness coincide.
\end{enumerate} 

\begin{theorem}
	Compact subsets of a submetric space are $G_{\delta}$ sets.
\end{theorem}  
\bproof
	Let $K \subset \cX$ be $\tau$-compact and let $d$ be a metric on $\cX$, compatible with $\tau$. Then $K$ is also $\tau_d$-compact, hence $\tau_d$-closed and therefore 
	\[ K = \bigcap_{n=1}^{\infty} \{ x\in \cX\,;\, d(x, K) < (1/n)\}.\]  
\eproof

\begin{theorem}\label{thm_compmetr}
	Compact subsets of a submetric space are metrizable. 
\end{theorem}
\bproof
	Let $(\cX,\tau)$ be a submetric space. We have to show that if $K$ is $\tau$-compact, then $\tau$ on $K$ is generated by a metric. We have an obvious candidate for this metric: let $\tau_{d}$ be the metric topology coarser than $\tau$. Then $\tau_{d}$ restricted to a $\tau$-compact set $K$ is a Hausdorff topology coarser than $\tau$. Repeating the reasoning from Example \ref{Ex:Tych} (in fact - applying  \cite[Theorem 3.1.13]{Eng89}) we obtain that  $\tau$ and $\tau_{d}$ coincide on $K$.
\eproof

\begin{remark}
	Note that the metric topology on a $\tau$-compact set $K$ is {\em independent} of the chosen particular metric $d$ (only compatibility with $\tau$ is required). This is the reason that, while talking on a submetric space $(\cX, \tau)$, we do not relate any explicit metric with it.     
\end{remark}

\begin{remark}
	The metrizability of compact subsets is crucial for the submetric space version of the almost sure Skorokhod  representation theorem for weakly convergent sequences \cite{Jak97a}. 
	Due to its ease of applicability, this version has become a commonly used tool in studying existence problems for certain classes of stochastic partial differential equations.
\end{remark}

In what follows, we will  need a simple lemma, confirming our metric intuition.
\begin{lemma}\label{Lem:Simp}
	Let $F \subset \cX$ be $\tau$-closed. If $\{x_n\} \subset F$ and $x_n \intau x_0$, then $x_0 \in F$.
\end{lemma}
\bproof
	Suppose that $x_0 \not\in F$. As $F^c$ is $\tau$-open, it follows from the convergence $x_n \intau x_0$ that for sufficiently large $n$ we have $x_n \in F^c$. A contradiction.
\eproof   

\begin{theorem} \label{thm_seqcomp}
	A subset of a submetric space is compact if, and only if, it is {\em sequentially compact}.
\end{theorem}
\bproof
	We know from Theorem \ref{thm_compmetr} that compact sets in a submetric space are metrizable, hence sequentially compact. So assume that $K \subset \cX$ is sequentially $\tau$-compact. Let $d$ be any metric on $\cX$, compatible with $\tau$.
	It follows that $K$ is sequentially $\tau_{d}$-compact, hence $\tau_{d}$-compact. 
	
	We shall show first that for any $\tau$-closed $F$ the intersection $F\cap K$  is sequentially $\tau$-compact, hence  $\tau_{d}$-compact. Let $x_n\in F\cap K, n\in\GN$. There exists a subsequence $(x_{n_k})$ $\tau$-convergent to $x_0\in K$. Since 
	$F$ is $\tau$-closed, $x_0 \in F$ (Lemma \ref{Lem:Simp}!) as well and the claim follows. 
	
	Thus $F\cap K$ is $\tau_{d}$-compact, hence 
	$(F\cap K)^c$ is $\tau_{d}$-open. Let $G$ be $\tau$-open. Then $G\cap K = \big(G^c \cap K\big)^c \cap K$ and so for each $G\in \tau$ there exists  $\widetilde{G}\in \tau_{d}$ such that 
	\[ G \cap K = \widetilde{G}\cap K.\]
	Finally,  let $K \subset \cup_{i\in\GI} U_i$, where $U_i,\ i\in\GI$, are $\tau$-open. Let $\widetilde{U}_i,\ i\in\GI$, be $\tau_{d}$-open and such that
	$U_i \cap K = \widetilde{U}_i\cap K$. Then also $K \subset \cup_{i\in\GI} \widetilde{U}_i$. Since $K$ is $\tau_{d}$-compact, there exists a finite set $\GI_0 \subset \GI$ such that $K \subset \cup_{i\in\GI_0} \widetilde{U}_i$. Then
	\[ K = \cup_{i\in\GI_0} \widetilde{U}_i \cap K = 
	\cup_{i\in\GI_0} U_i \cap K \subset \cup_{i\in\GI_0} U_i.\]
\eproof

\begin{remark}\label{Rem:compare}
	Suppose we have two comparable topologies $\tau_1$ and $\tau_2$ on $\cX$: $\tau_1 \subset \tau_2$. If $(\cX,\tau_1)$ is submetric, then so is $(\cX,\tau_2)$. 
	
	Further, suppose that any $\tau_1$-compact subset is also $\tau_2$-compact. This implies that the families $\cK_{\tau_i} = \{ K\,|\, \text{ $K$ is  $\tau_i$-compact}\}$, $i=1,2$, coincide. 
	
	These simple facts can have important consequences in limit theorems for stochastic processes. In order  to check \emph{the uniform tightness condition} of some family $\{X_{i}\}_{i\in\GI}$ of random elements with values in the submetric space $(\cX,\tau_1)$, one verifies whether for every $\varepsilon > 0$ there exists a $\tau_1$-compact set  $K_{\varepsilon}$ such that
	\begin{equation}\label{eq:ut}
		\bP\big( X_i \in K_{\varepsilon} \big) > 1 - \varepsilon,\qquad i\in\GI.
	\end{equation}
	When $\cK_{\tau_1} = \cK_{\tau_2}$, equation (\ref{eq:ut}) yields the uniform tightness of $\{X_{i}\}_{i\in\GI}$ in $(\cX,\tau_2)$! Moreover, since $\tau_2 \supset \tau_1$, the larger topology $\tau_2$ provides a wider class of continuous functions, which can make the resulting limit theorem substantially richer.
	
	It is interesting that one can identify the maximal topology $\tau_2$ that satisfies the property $\cK_{\tau_1} = \cK_{\tau_2}$.
\end{remark}

\begin{definition} \label{tauesclosed}
	Let $(\cX,\tau)$ be a Hausdorff topological space. Say that $F\subset \cX$ is $\tau_s$-closed if limits of $\tau$-convergent sequences of elements of $F$ remain in $F$, i.e. 
	\begin{equation}\label{eq:closed}
		\text{if $x_n \in F, \ n=1,2,\ldots$, and $x_n\intau x_0$, then $x_0\in F$}. 
	\end{equation}
	The topology given by $\tau_s$-closed sets is called the sequential topology generated by $\tau$ and will be denoted by $\tau_s$. 
\end{definition}

\begin{theorem} \label{thm_seqtopol}
	Let  $(\cX,\tau)$ be a Hausdorff topological space. Then
	\begin{enumerate}
		\item $\tau \subset \tau_s$ (i.e. $\tau_s$ is finer than $\tau$).
		\item $x_n \intaues x_0$ if, and only if, $x_n \intau x_0$.
	\end{enumerate}
	In particular, $\tau_s$ is Hausdorff. 
\end{theorem}
\bproof If $F$ is $\tau$-closed, then it is $\tau_s$-closed by Lemma \ref{Lem:Simp}. 
	
	It follows that comparing to $\tau$, there are ``more'' $\tau_s$-open neighborhoods and therefore convergence of sequences in $\tau_s$ is ``more demanding''. Hence $x_n \intaues x_0$ implies $x_n \intau x_0$.
	
	To prove the converse implication, assume that
	$x_n \notintaues x_0$. This means that there exists a $\tau_s$-open set $U\ni x_0$ and a subsequence $\{x_{n_k}\}$ such that $x_{n_k} \in U^c$, for all $k\in\GN$. Since $U^c$ is $\tau_s$-closed, any $\tau$-limit point of $\{x_{n_k}\}$ must belong to $U^c$ and therefore cannot be $x_0$. Consequently,  $x_n \notintau x_0$.
	
	The last statement stems from (i).
\eproof

\begin{corollary} \label{compequiv} Let $(\cX,\tau)$ be a Hausdorff topological space. A subset $K \subset \cX$ is $\tau$-sequentially compact if, and only if, it is $\tau_s$-sequentially compact. 
\end{corollary}

\begin{theorem}\label{thm_finest}
	Let $(\cX,\tau)$ be a submetric space. The topology $\tau_s$ is the finest topology among topologies which are finer than $\tau$ and have the same compact subsets as $\tau$.
\end{theorem}
\bproof
	Since $\tau \subset \tau_s$, $\tau_s$ is submetric. By Theorem \ref{thm_seqcomp} for both $\tau$ and $\tau_s$, a subset $K$ is compact if and only if it is sequentially compact. By Corollary \ref{compequiv}, the sequential compactness means the same for both $\tau$ and $\tau_s$. It follows that $\tau_s$ has the desired property. 
	
	Now suppose that $\tau \subset \tau'$ and that $\cK_{\tau} = \cK_{\tau'}$. Let $F$ be  $\tau'$-closed. We will show that it is $\tau_s$-closed and this will imply $\tau' \subset \tau_s$.
	Take $x_n\in F, n\in\GN,$ and assume that $x_n \intau x_0$. Set $K = \{x_n\,;\, n\in\GN\} \cup \{x_0\}$. It is $\tau$-compact, hence $\tau'$-compact. In particular, $K \cap F$ is also $\tau'$-compact, hence $\tau$-compact. Therefore $x_0 \in F\cap K$ and $x_0 \in F$. This means that $F\in \tau_s$.  
\eproof
\begin{remark}
	The topology $\tau_s$ can be substantially finer than the initial submetric topology $\tau$. For instance, if we start with the weak topology $\tau_w$ on a separable Hilbert space (Example \ref{Ex:Hilb}), then the sequential topology $(\tau_w)_s$ is the so-called bounded topology on $\GH$ (see, e.g., \cite[Section 5.5]{DunSch58}).
\end{remark}

\begin{remark}\label{Rem:Intuition}
	Theorem \ref{thm_seqcomp} and relation (\ref{eq:2}) provide a modeling framework, which is unavailable not only within the metric regime, but also in the broader context of completely regular spaces with metrizable compacts \cite{SmFo76}. Specifically, we refer to sequential spaces generated by $\cL$-convergences.  We will focus only on intuition and will not develop the formal theory here, referring instead to \cite[Section 6]{Jak18} for all necessary concepts and results.
	
	Imagine a space $\cX$ within which we are given a family $\cC$ of checkable conditions ``suspected" for forming criteria of compactness in some topology. Suppose we are able to find a convergence of sequences $\conver$ of elements of $\cX$ such that any set $C$ satisfying the conditions $\cC$ is relatively compact with respect to $\conver$. If there exists a countable family $\{f_i\}$ of {\em sequentially} continuous functions on $(\cX,\conver)$ which separates points in $\cX$, then (\ref{eq:2}) defines a compatible metric that converts our space $(\cX,\conver)$ into a submetric space!
	
	Although the above scenario appears technically demanding, it was successfully implemented?using conditions $\cC$ from Example~\ref{Ex:Stop}?in the construction of the so-called $S$ topology on the Skorokhod space \cite{Jak97b}. This topology, which is weaker than the more popular Skorokhod $J_{1}$ and $M_{1}$ topologies, has found significant applications, most recently in optimal martingale transport on the Skorokhod space (see, e.g., \cite{CKPS21, LiNe19}). Moreover, even the awareness of the nuances associated with the construction of submetric spaces from sequential spaces is slowly spreading (see, e.g., \cite{JAR26}). 
\end{remark}

\section{C{\`a}dl{\`a}guity and relative compactness}\label{Sec:relcomp}
Let us consider a c\`adl\`ag function $x : [0,T] \to \cX$ with values in a topological space $(\cX,\tau)$. The set of values $V_x = \{x(t)\,;\,t \in [0,T]\}$ is \emph{relatively compact}. Indeed, if we take any sequence $\{x(t_n)\} \subset V_x$, then we may find a monotone subsequence $\{t_{n_k}\}$ converging to some $t_0 \in [0,T]$, and along this subsequence either $x(t_{n_k}) \to x(t_0)$ or $x(t_{n_k}) \to x(t_0-)$, depending on whether $\{t_{n_k}\}$ is non-decreasing or non-increasing.

Returning to the question from the Introduction regarding the compactness of the extended set of values $ K_x = \{ x(t)\, ; \, t\in [0,T]\} \cup \{ x(t-)\, ; \, t\in (0,T]\}$ when $(\cX,\tau)$ is submetric, we see that it can be reduced to two intermediate questions about submetric spaces:
\begin{enumerate}
	\item What is the closure of a relatively compact set?
	\item Is the closure of a relatively compact set compact?
\end{enumerate}

The answer for the first question is quite natural.

\begin{theorem} \label{Th:closure}
	In submetric spaces, the closure of a relatively compact subset $J$ coincides with the metric closure and consists of the limits of convergent sequences of elements of $J$. 
\end{theorem}
\bproof
	Let $J$ be a relatively compact subset of $(\cX,\tau)$ and let $\overline{J}$ be the set described in the theorem: the set of $\tau$-limits of all convergent sequences of elements of $J$. Let ${\overline{J}}^{\tau}$ be the $\tau$-closure of $J$. 
	
	Let $d$ be a metric compatible with $\tau$ and let ${\overline{J}}^{\tau_d}$ be the closure of $J$ with respect to $d$. It consists of $d$-limits of sequences of elements in $J$. We claim that ${\overline{J}}^{\tau_d} = \overline{J}$. 
	Indeed, if $\{x_n\} \subset J$ and $x_n\intau x_0$, then $d(x_n,x_0) \to 0$ due to the $\tau$-continuity of $d$. Hence $\overline{J} \subset {\overline{J}}^{\tau_d}$. 
	Conversely, take  $x_0 \in {\overline{J}}^{\tau_d}$. One can find  a sequence $\{x_n\} \subset J$ such that $d(x_n,x_0) \to 0$. By the relative compactness of $J$, there exists a $\tau$-convergent subsequence $x_{n_k}\intau y_0 \in \overline{J}$. Again, by the continuity of $d$, we have $d(x_{n_k},y_0) \to 0$. It follows that $y_0 = x_0$ and so $\overline{J} = {\overline{J}}^{\tau_d}$.
	
	On the other hand, as $\tau_d \subset \tau$, we have $\overline{J}^{\tau} \subset \overline{J}^{\tau_d}$. Since $\overline{J}^{\tau_d} = \overline{J}$, it is enough to show that $\overline{J} \subset \overline{J}^{\tau}$. Let $F$ be any $\tau$-closed set containing $J$. For each $x_0 \in \overline{J}$, one can find a sequence $\{x_n\} \subset J \subset F$ such that $x_n \xrightarrow{\tau} x_0$. By Lemma \ref{Lem:Simp}, we have $x_0 \in F$, which implies that $\overline{J} \subset F$. Consequently, $\overline{J} \subset \overline{J}^{\tau}$.
	
\eproof 
\begin{corollary}\label{Cor:relcomp}
	In a submetric space $(\cX,\tau)$, if both $J$ and $\overline{J}$ are relatively $\tau$-compact, then $\overline{J}$ is $\tau$-compact.
\end{corollary}
\bproof
	By Lemma \ref{Lem:Simp} $\overline{J}$ is sequentially $\tau$-compact. By Theorem \ref{thm_seqcomp} it is $\tau$-compact.
\eproof

\begin{remark}
	It follows from the above theorem that in a submetric space, for a c\`adl\`ag function  $x : [0,T] \to \cX$, the closure of $V_x = \{x(t)\,|\,t \in [0,T]\}$ is $ K_x = \{ x(t)\, ; \, t\in [0,T]\} \cup \{ x(t-)\, ; \, t\in (0,T]\}$. Is it compact?
\end{remark}
\begin{example}\label{Ex:Jans}
	This is the example from \cite[Example 2.6]{Jan26}. Let $(\cX,\tau)$ be the Smirnov space described in Example \ref{Ex:Smir}. Recall that $Z = \{ 1/n\,|\,n\in\GN\}$ is the non-typical closed set in $\GR^1$. Let $x : [0,1]\to \cX$ be defined by
	\begin{equation}
		x(t) = \begin{cases} 
			0, &\text{if $x=0$};\\
			\frac{1}{2}(t+\frac{1}{n}), &\text{if $t \in [\frac{1}{n+1},\frac{1}{n})$}, n\in \GN;\\
			1, &\text{if $x=1$}.
		\end{cases}
	\end{equation}
	Note that $V_x \cap Z = \emptyset $, while $K_x = [0,1]\supset Z$.  Since no subsequence of distinct elements of $Z$ is convergent, and $(\cX, \tau)$ is a submetric space, the set $K_x$ \emph{is not} $\tau$-compact  by Theorem \ref{thm_seqcomp}.
\end{example}
\begin{remark}
	The Smirnov space is, in several aspects, pathological among submetric spaces.  In what follows we shall restrict our attention to \emph{regular submetric spaces}. 	
\end{remark}
\begin{definition} \label{Def:regsubmspace}
	We shall say that a topological space $(\cX,\tau)$ is \emph{a regular submetric space} if it is a submetric space and the closure $\overline{J}$ of each relatively $\tau$-compact space $J$ is $\tau$-compact.
\end{definition}
\begin{corollary}
	If $(\cX,\tau)$ is a regular submetric space, then for every c\'adl\`ag function $x :[0,T] \to \cX$ the extended set of values $K_x = \{ x(t)\, ; \, t\in [0,T]\} \cup \{ x(t-)\, ; \, t\in (0,T]\}$ is $\tau$-compact.
\end{corollary}
\begin{remark} The existence of non-regular submetric spaces has a significant impact on the theory of limit theorems for stochastic processes. Let us return to the notion of the uniform tightness, discussed in Remark \ref{Rem:compare}. Recall that a family $\{X_{i}\}_{i\in\GI}$ of random elements with values in a submetric space $(\cX,\tau)$ is uniformly $\tau$-tight if  for every $\varepsilon > 0$ there exists a $\tau$-compact set  $K_{\varepsilon}$ such that
	\[
	\bP\big( X_i \in K_{\varepsilon} \big) > 1 - \varepsilon,\qquad i\in\GI.
	\]
	If $(\cX,\tau)$ is a regular submetric space, the above condition can be replaced with 
	\begin{equation}\label{eq:nonreg}
		\bP\big( X_i \in J_{\varepsilon} \big) > 1 - \varepsilon, \qquad i\in\GI,
	\end{equation}
	where $J_{\varepsilon}$ is merely relatively $\tau$-compact. However, if $(\cX,\tau)$ \emph{is not} a regular submetric space,  condition (\ref{eq:nonreg}) is strictly weaker than the uniform $\tau$-tightness.
\end{remark}  

What are the sufficient conditions  for $(\cX,\tau)$ to be a regular submetric space? We shall provide one verifiable  and two theoretical ones.

Recall that a functional $h :\cX \to \GR^1$ is {\em sequentially lower semicontinuous} if
\[\liminf_{n\to\infty} h(x_n) \geq f(x_0),\]
whenever $x_n \intau x_0$. 

\begin{theorem}\label{Th:lowersemi}
	Let $(\cX,\tau)$ be a submetric space. Suppose we are given a family $\{h_i : \cX \to [0,+\infty)\}_{i\in\GI}$ of {\em sequentially lower semicontinuous functionals} on $(\cX,\tau)$ such that
	$J \subset \cX$ is relatively $\tau$-compact if, and only if, 
	\[ \sup_{x\in J} h_i(x) < +\infty,\qquad i\in\GI.\]
	
	Then $(\cX,\tau)$ is a regular submetric space. 
\end{theorem}

\bproof
	Let $J$ be relatively $\tau$-compact and let $x_0 \in {\overline{J}}^{\tau}$. By Theorem \ref{Th:closure}, there exists a sequence $\{x_n\} \subset J$ such that
	$x_n \intau x_0$, hence 
	\[ h_i(x_0) \leq \liminf_{n\to\infty} h_i(x_n) \leq \sup_{x\in J} h_i(x),\qquad i\in\GI.\]  
	It follows that ${\overline{J}}^{\tau}$ is relatively $\tau$-compact, hence $\tau$-compact (Corollary \ref{Cor:relcomp})
\eproof

\begin{example}\label{Ex:Hilb2}
	Let us return to Example \ref{Ex:Hilb}. It is well-known  that a subset $J \subset \GH$ is relatively $\tau_w$-compact if, and only if, it is norm-bounded, i.e., 
	\[ \sup_{x\in J} \|x\| < +\infty.\]
	(See, e.g., the Eberlein-\v{S}mulyan theorem, \cite[Section V.6]{DunSch58}).
	Furthermore, the norm is sequentially lower semicontinuous.
	Indeed, if $x_n \intauw x_0$, then for every $i\in \GN$ we have $\langle x_n,e_i\rangle \to \langle x_0,e_i\rangle$ and by the Fatou lemma
	\[ \liminf_{n\to\infty} \|x_n\|^2 = \liminf_{n\to\infty} \sum_{i=1}^{\infty} \langle x_n,e_i\rangle^2 \geq \sum_{i=1}^{\infty} \langle x_0,e_i\rangle^2 = \|x_0\|^2.\]
	It then follows that $(\GH,\tau_w)$ is a regular submetric space.
\end{example}
\begin{example}[$S$ topology]\label{Ex:Stop}
	The $S$ topology on the Skorokhod space $\GD = \GD([0,T])$ of c\`adl\`ag functions $x : [0,T] \to \GR^1$ was constructed in  \cite{Jak97b}. We will define the set $\cC$ of conditions (see Remark \ref{Rem:Intuition}) implying the relative compactness in the $S$ topology.
	
	Let us first introduce some relevant notation.
	\begin{enumerate}
		\item $\|x\|_{\infty}$ is the uniform norm on $\GD$, i.e., $ \|x\|_{\infty} = \sup_{t\in [0,T]} |x(t)|$.
		\item 
		For $a < b$, $N^{a,b}(x)$ is \emph{the number of up-crossings of levels $a$ and $b$ by function $x\in \GD$}. In other words, $N^{a,b}(x)$ is the largest integer $k$ such that there are numbers  $0 \leq t_1 < t_2 <t_3 < \ldots < t_{2k-1} < t_{2k} \leq T$ satisfying $x(t_{2i - 1}) < a,\ x(t_{2i}) > b$, for each $i=1,2,\ldots, k$.
		\item For $\eta > 0$, $N_{\eta}(x)$ is \emph{the number of $\eta$-oscillations} of $x\in \GD$ on $[0,T]$. This means that $N_{\eta}(x)$ is the largest ineger $k$ such that there are numbers $0 \leq t_1 < t_2 \leq t_3 <  \ldots \leq  t_{2k-1} < t_{2k} \leq T$ satisfying $ \big|x(t_{2i}) - x(t_{2i - 1})\big| > \eta$, for each $i=1,2,\ldots, k$.
		\item $\|v\|(T)$ is \emph{the total variation of $v$ on $[0,T]$}:
		\[
		\|v\|(T)
		= \sup \big\{|v(0)| + \sum_{i=1}^{m} |v(t_i) - v(t_{i-1})|\big\}, \]
		where the supremum is taken over all 
		$0=t_0 < t_1 < \ldots < t_m = T,\ m\in\GN$.
		\item $\bV = \big\{ x \in \GD\,;\, \|x\|(T) < +\infty\big\}$.
		\item If $v_{n,\varepsilon}, v_{0,\varepsilon} \in \bV$, then  $v_{n,\varepsilon} \Rightarrow v_{0,\varepsilon}$ means that 
		\begin{equation}\label{evconvclear}
			\int_{[0,T]} f(t) \,dv_{n,\varepsilon}(t) \to \int_{[0,T]} f(t)\,dv_{0,\varepsilon}(t),
		\end{equation}
		for each continuous function $f : [0,T] \to \GR^1$. In particular, setting $f(t) \equiv 1$ we get
		\begin{equation}\label{vone}
			v_{n,\varepsilon}(T) \to v_{0,\varepsilon}(T).
		\end{equation}
	\end{enumerate}

	The $S$ topology is a sequential topology defined via the $S$-convergence. 
	We shall write $x_n \ines x_0$
	if for every $\varepsilon > 0$ one can
	find elements $v_{n,\varepsilon}\in \GV$, $n=0,1,2,\ldots $ which are
	$\varepsilon$-uniformly close to $x_n$'s and weakly-$*$ convergent:
	\begin{eqnarray}
		\|x_n - v_{n,\varepsilon}\|_{\infty} \leq \varepsilon,&&\ n = 0, 1, 2,
		\ldots, \label{vclose}\\
		v_{n,\varepsilon} \Rightarrow v_{0,\varepsilon},&& \text{as $n\to\infty$}.\label{vconv}
	\end{eqnarray}
	Let us consider the following conditions describing properties of a subset $K\subset \GD$.
	\begin{align}
		\label{2e1}
		\sup_{x\in K}
		\|x\|_{\infty} &< +\infty. \\
		\label{2e2}
		\sup_{x\in K} N^{a,b}(x) &< +\infty,\quad \text{for all $a < b$}.\\
		\label{2e3}
		\sup_{x\in K} N_{\eta}(x) &< +\infty, \quad \text{for every $\eta > 0$}.
	\end{align}
	Then one can prove that either of the equivalent sets of conditions \ref{2e1} $+$ \ref{2e2} and  \ref{2e1} $+$ \ref{2e3} gives a criterion of relative $S$-compactness for $K$.
	
	Moreover, \cite[Corollary 2.10]{Jak97b} states that $\|\cdot\|_{\infty}$, $N^{a,b}(\cdot)$ for $a < b$, and $N_{\eta}(\cdot)$ for $\eta > 0$ are sequentially lower semicontinuous with respect to  $x_n \ines x_0$.
	
	We have thus found at least two candidates  
	for the family $\{h_i\}_{i\in\GI}$ in Theorem \ref{Th:lowersemi}.
\end{example}

\begin{remark}
	More general (but still applicable) sufficient conditions for $(\cX,\tau)$ to be a regular submetric space can be found in \cite[Theorem 6.3]{Jak00}.
	
	In fact, the ideal situation would be if we could prove that a ``regular topological space'' that is also a ``submetric space'' yields a ``regular submetric space''.
	While we are unable to obtain such a result, we  provide two partial results in this direction. 
\end{remark}
Recall that $(\cX,\tau)$ is a regular topological space if for each $\tau$-closed set $F$ and each $x_0 \notin F$ there are $\tau$-open sets $U$ and $V$ such that $F \subset U$, $x_0 \in V$ and $U\cap V = \emptyset$.
\begin{proposition}
	Let $(\cX,\tau)$ be a submetric space and a regular topological space. Suppose that either of the following two conditions is satisfied: 
	\begin{enumerate}
		\item $(\cX,\tau)$ is sequential.
		\item $(\cX,\tau)$ is first-countable (i.e., at each point, there is a countable local basis).
	\end{enumerate}
	Then $(\cX,\tau)$ is a regular submetric space. 
\end{proposition} 
\bproof

	Suppose that $(\cX,\tau)$ is a submetric space, but not a regular submetric space.  Let $J$ be a relatively $\tau$-compact set such that $\overline{J}$ {\em is not} $\tau$-compact. By Corollary \ref{Cor:relcomp}, the closure $\overline{J}$ is not relatively $\tau$-compact. It follows that there exists a sequence $\{ x_n\} \subset \overline{J} \setminus J$, {\em no subsequence of which is $\tau$-convergent}. 
	
	First let us assume that $\tau$ is sequential; that is, a subset of $\cX$ is $\tau$-closed if it contains the limits of all its convergent sequences.   
	In particular, the set $F$ consisting of the elements of the sequence $\{x_{n}\}$ is closed (since no subsequence of this sequence converges), as is any subset of $F$. 
	
	Now, chose a compatible metric $d$ and note that  $\overline{J}$ is $\tau_d$-compact. Therefore there exists a subsequence $\{x_{n_k}\}$ such that $d(x_{n_k},x_0)\to 0$, for some $x_0 \in \overline{J}$. By eliminating repetitions, we may and do assume that $x_0 \not\in F$. 
	
	Suppose that $(\cX,\tau)$ is a regular topological space. Then one can find $\tau$-open sets $U$ and $V$ such that $\{x_n\} \subset U$, $x_0 \in V$ and $U \cap V = \emptyset$.  
	
	For each $k\in\GN$, let $x_k'$ be an element of $J$ satisfying $x_k' \in U$ and $d(x_{n_k},x_k') < 1/k$. Then $d(x_k',x_0) \to 0$ and $\{x_k'\}$ is relatively $\tau$-compact; hence $x_k' \intau x_0$, which implies $x_k' \in V$ for $k \geq k_0$. However, this is impossible due to $U\cap V =\emptyset$. It follows that the sequential submetric space $(\cX,\tau)$ that is not a regular submetric space cannot be a regular topological space.
	
	Next, let us assume that $(\cX,\tau)$ is a first-countable submetric space that is, once again, not a regular submetric space.
	
	Exactly as before, we can find a sequence $\{x_k\}$ with no  $\tau$-convergent subsequence, and a point $x_0 \in \overline{J}$ such that $d(x_k,x_0) \to 0$. If the set $F$ consisting of the elements of the sequence $\{x_k\}$ is $\tau$-closed, we can repeat the previous reasoning to deduce that $(\cX,\tau)$ cannot be regular. 
	
	Thus, let us assume that $F$ {\em is not} $\tau$-closed. Its closure $\overline{F}$ must be  $F \cup \{x_0\}$, as it cannot be larger than the $\tau_d$-closure. The fact that $x_0 \in \overline{F}$ implies (is equivalent to) that for every $\tau$-open set $V \ni x_0$, we have $V\cap F \neq \emptyset$. We then choose a countable neighborhood basis   $\{V_m(x_0)\}$ at $x_0$ and let $x_{k_m}$ be an element of $\{x_k\}$ that belongs to $\cap_{j=1}^m V_j(x_0)$. Clearly, $x_{k_m} \intau x_0$ which contradicts the defining property of $\{x_k\}$
\eproof  

\section{Back to the Hausdorff case: reversibly c\`adl\`ag functions}\label{Sec:Hausdorff}

Example \ref{Ex:Jans} (due to Janson \cite{Jan26}) disproved author's claim \cite[Proposition 1, stated without proof]{Jak86} that the extended set of values $K_x$ is compact for any c\`adl\`ag function $x:[0,T] \to (\cX,\tau)$ with values in a general Hausdorff space.
As announced in the Introduction, Janson \cite[Corollary 2.5]{Jan26} proved this result for regular topological spaces $(\cX,\tau)$. In his proof, he used so-called {\em split interval} (or {\em arrow space}) $\widehat{I}$, which is compact, totally disconnected, separable and first-countable topological space that is neither second-countable nor metrizable. See \cite[Chapter 9]{JaKa15} for more information on this fascinating space in the context of the Skorokhod space considered as a Banach algebra. 

Interestingly, filling a gap in the original (unpublished) proof of Proposition 1 in \cite{Jak86} leads to an elementary proof of a fact of independent interest, which remains valid in general Hausdorff spaces.    

\begin{definition} Let $(\cX,\tau)$ be a Hausdorff space and let $x: [0,T] \to \cX$ be a c\`adl\`ag function. Define 
	\begin{equation}\label{eq:revers}
		y_x(t) := x(t-),\quad \text{ if $t\in (0,T]$};\qquad y_x(0) = x(0).
	\end{equation}
	We will say that $x$ is {\em reversibly } c\`adl\`ag, if $y$ is c\`agl\`ad, i.e., left continuous on $[0,T]$ and with limits form the right at $t\in [0,T)$.
\end{definition}

\begin{theorem}\label{Th:revers} If $x$ is  reversibly c\`adl\`ag, then $K_x$ is $\tau$-compact in $\cX$.
\end{theorem}
\bproof
	For the sake of brevity, let us extend functions $x$ and $y_x$ to $\GR^1$ in a natural way
	\begin{align*}
		x(t) &= x(0),\quad \text{if $t < 0$};\quad x(t) = x(T), \quad \text{if $t > T$};\\
		y_x(t) &= y(0),\quad \text{if $t < 0$};\quad y_x(t) = y(T), \quad \text{if $t > T$}.
	\end{align*}
	Notice that the definition of $K_x$ is 
	independent of whether we consider $x$ as a function on $[0,T]$ or its extension to $\GR^1$.
	
	Let us consider an open cover of the set $K_x$:
	\begin{equation}\label{aje1}
		K_x \subset \bigcup_{\alpha\in A} U_{\alpha},
	\end{equation}
	We will find a finite subcover of $K_x$.
	
	We shall say that an {\em open} interval $I \subset \GR^1$ is ``good", if one can find $U_{\alpha}$ such that
	\[ \{x(t)\,;\, t \in I\} \subset U_{\alpha}.\]
	Let $S$ be the sum of all ``good" intervals. By the definition of the extended $x$, $(-\infty,0) \subset S$, $(T,+\infty) 
	\subset S$.
	
	Let us consider the set 
	\[ B = [0,T] \setminus S.\]
	We shall prove 
	\begin{lemma}\label{Lem1}
		$B$ is at most finite. 
	\end{lemma}
	
	\bproof[Proof of Lemma \ref{Lem1}]
		Suppose the contrary, i.e. one can find in $B$ 
		a sequence $t_1, t_2,\ldots$ of {\em distinct} points. Since $B \subset [0,T]$, there exists 
		a convergent subsequence $\{t_{n_k}\}$:
		\[ t_{n_k} \to t_0 \in [0,T],\ \text{ as } k\to\infty.\]
		Since $t_n$'s are distinct, at most one is equal to $t_0$ and so for $k \geq k_0$
		the subsequence $\{t_{n_k}\}$ can be split into two subsequences: $\{t_{n_k} > t_0\}$ 
		and $\{t_{n_k} < t_0\}$ (one of them may be empty). This means that the first subsequence
		\[ t_{n_{k'}} \to t_0 +,\]
		while the second 
		\[ t_{n_{k''}} \to t_0 -.\]
		
		In the first case, $t_0 < 1$  and by (\ref{aje1}), there exists $\alpha \in A$ such that $x(t_0) \in U_{\alpha}$. 
		By the right continuity of $x$ at $t_0$ there exists $\varepsilon > 0$ such that
		\[ x(s) \in U_{\alpha}, \ s \in [t_0,t_0+\varepsilon).\]
		Hence $(t_0,t_0+\varepsilon)$ is a ``good'' interval. 
		But for sufficiently large $k'$, 
		$t_{n_{k'}} \in (t_0,t_0+\varepsilon) \subset S$, what is a contradiction, since $\{t_n\} \subset [0,T] \setminus S$.
		
		In the second case, $t_0 > 0$ and by (\ref{aje1}), there exists   
		$\beta \in A$ such that $x(t_0-) \in U_{\beta}$. 
		By the existence of the limit from the left at $t_0$, there exists $\delta > 0$ such that
		\[ x(s) \in U_{\beta}, \ s \in (t_0 - \delta, t_0).\]
		Hence $(t_0-\delta, t_0)$ is a ``good'' interval. 
		But for sufficiently large $k''$, 
		$t_{n_{k''}} \in (t_0-\delta,t_0) \subset S$, what is a contradiction, since $\{t_n\} \subset [0,T] \setminus S$.
		
		Hence $B$ is finite.
	\eproof
	
	\noindent{\em Proof of Theorem \ref{Th:revers} - continued.}
	
	By Lemma \ref{Lem1}, let  $B = \{ t_1, t_2, \ldots, t_m\}$.
	By (\ref{aje1}), for every $j=1,2,\ldots,m$ there exist: $\alpha_j,\beta_j \in A$ and 
	$\varepsilon_j > 0, \delta_j > 0$ such that
	\begin{eqnarray*}
		x(t_j) \in U_{\alpha_j},&& x(s) \in U_{\alpha_j}, s\in (t_j,t_j + \varepsilon_j),\\
		x(t_j-) \in U_{\beta_j} && x(s) \in U_{\beta_j}, s\in (t_j-\delta_j,t_j).
	\end{eqnarray*}
	
	Let us consider 
	\[ B_1 = [0,T] \setminus \Big( \bigcup_{j=1}^{m} (t_j-\delta_j,t_j+\varepsilon_j)\Big) \subset
	[0,T] \setminus B \subset S.\]
	$B_1$ is compact, hence we can find a finite set of ``good'' intervals $(a_l,b_l), \ l=1,2,\ldots, q$ and 
	$\gamma_1,\gamma_2,\ldots,\gamma_q \in A$ such that
	\[ x(s) \in U_{\gamma_l},\ s\in (a_l,b_l), l = 1,2, \ldots, q,\]
	and
	\[ B_1 \subset \bigcup_{l=1}^{q} (a_l,b_l).\]
	
	Hence the sum 
	\[ \bigcup_{j=1}^m U_{\alpha_j} \cup \bigcup_{j=1}^{m} U_{\beta_j} \cup \bigcup_{l=1}^{q} U_{\gamma_l}
	\]
	provides a finite covering of $\{ x(t)\,;\, t\in [0,T]\}$ {\em plus}, perhaps,  some $x(t-)$. Example \ref{Ex:Jans} shows that some $x(t-)$ may be not included in this subcover.
	Therefore we repeat the procedure for $y_x(t)$ and get  another finite covering 
	\[ \bigcup_{u=1}^{m'} U_{\alpha_u'} \cup \bigcup_{u=1}^{m'} U_{\beta_u'} \cup \bigcup_{v=1}^{q'} U_{\gamma_v'}
	\]
	of $\{ y(t)\,;\,t\in [0,T]\}$ plus some $y(t+) = x(t)$. Merging these two families we obtain a finite cover of $K_x$.
\eproof

\begin{example}[Example \ref{Ex:Jans} revisited]
	Janson's example shows that not all c\`adl\`ag functions are reversibly c\`adl\`ag.
\end{example}

\begin{corollary}
	If $(\cX,\tau)$ is a completely regular space and $x$ is a c\`adl\`ag function with values in $\cX$, then
	$x$ is reversibly c\`adl\`ag. In particular, the set $K_x$ is $\tau$-compact.
\end{corollary}

\bproof
	Let us recall that a completely regular topology $\tau$ on $\cX$ is generated by family of  pseudometrics $\{d_i\}_{i\in\GI}$, separating points (i.e., if $a,b \in \cX, a \neq b $ then $d_i(a, b) > 0$ for some $i\in\GI$) and is directed (i.e., for any pair $d_i$, $d_j$ there is a $d_k$ such that $\max(d_i,d_j) \leq d_k$). In particular, the convergence of sequences in $(\cX,\tau)$ is equivalent to convergence in each pseudometric $d_i$.
	See, e.g., \cite[Theorem 8.1.20]{Eng89}.
	
	Let $x: [0,T] \to \GR^1$ be a c\'adl\`ag function and consider $y_x$ given by  (\ref{eq:revers}). Let $t_n \nearrow t_0 \in (0,T]$. We want to prove  the left-continuity  of $y_x$, that is, 
	\[ y_x(t_n) = x(t_n-) \intau y_x(t_0) = x(t_0-).\]
	Fix a pseudometric $d_i$ and choose $t_n'$ such that $t_n -1/n < t_n' < t_n$ and $d_i(x(t_n'), x(t_n-)) < 1/n$. Then
	\[ d_i( y_x(t_n), y_x(t_0)) \leq d_i(x(t_n-),x(t_n'))  + d_i( x(t_n'), x(t_0-))\to 0.
	\]
	In a similar way, we can verify that $y_x$ admits limits from the right.
\eproof



\begin{thebibliography}{99}
	
	\bibitem{JAR26}
	Abi Jaber, E., Attal, E. and Rosenbaum, M.: From hyper roughness to jumps as $H\to -1/2$. \emph{Ann. Appl. Probab.}, {\bf 36 No. 4}, (2026), 3635--3660.
	
	\bibitem{CKPS21}
	Cheridito, P.,  Kiiski, M., Pr{\"{o}}mel, D.J. and  Soner, H.M.:
	Martingale optimal transport duality, \emph{Math. Ann.}, {\bf 379, No. 3-4}, (2021), 1685--1712.
	
	\bibitem{DunSch58}
	Dunford, N. and Schwartz, J.T.,(With the assistance of W.B. Bade and R.G. Bartle): {\bf Linear Operators. Part I. General theory}, Interscience Publishers, New York 1958.
	
	\bibitem{Eng89}
	Engelking, R.: {\bf General Topology}. Helderman, Berlin 1989.
	
	\bibitem{Gru84}
	Gruenhage, G.: Generalized metric spaces, in: K. Kunen \& J.E.~Vaughan, Eds.,  {\bf Handbook of set-theoretic topology}, North-Holland, Amsterdam
	1984, 423--502.
	
	
	\bibitem{Jak86} Jakubowski, A.: On the Skorokhod Topology. \emph{Ann. Inst. H. Poincar{\'e} Probab. Statist.} \textbf{22}, (1986), 263--285.
	
	\bibitem{Jak97a} Jakubowski, A.: The almost sure Skorokhod representation for subsequences in nonmetric spaces. \emph{Teor. Veroyatnost. i Primenen.}
	\textbf{42}, (1997), 209--216. (English transl.: {\em Theory Probab. Appl.} \textbf{42}, (1997), 167--174.) 
	
	\bibitem{Jak97b}
	Jakubowski, A.: A non-Skorohod topology on the Skorohod space. \emph{ Electr. J. Prob.} \textbf{2}, (1997), No 4, 1-21. 
	
	\bibitem{Jak00}
	Jakubowski, A.: From convergence of functions to convergence of stochastic processes. On Skorokhod's sequential approach to convergence in distribution, in: V. Korolyuk, N. Portenko \& H. Syta, Eds.: {\bf Skorokhod's Ideas in Probability Theory}. \emph{Insitute of Mathematics, National Academy of Sciences of Ukraine}, Kyiv 2000, 179--194. 
	
	\bibitem{Jak18}
	Jakubowski, A.: New characterizations of the S topology on the Skorokhod space, \emph{Electron. Commun. Probab.}, \textbf{23}, (2018), No.~2, 1--16.
	
	
	\bibitem{Jak23} Jakubowski, A.: Probability on submetric spaces, \emph{Ann. Math. Sil.} \textbf{37}, (2023), 138--148.
	
	\bibitem{Jar81}
	Jarchow, H.: {\bf Locally convex spaces}, B.G. Treubner, Stuttgart 1981. 
	
	\bibitem{Jan26} Janson, S.: On the Skorohod topology for functions with values in a completely regular space. \emph{Theor. Probability and Math. Statist.} {\bf 114} (2026)
	DOI: https://doi.org/10.1090/tpms/1254\ ArXiv: {2511.0911}.
	
	\bibitem{JaKa15}
	Janson, S. and Kaijser, S.: Moments of Banach Space Valued Random Variables. \emph{Mem. Amer. Math. Soc.}
	{\bf 238} (2015), no. 1127.
	
	\bibitem{LiNe19}
	Liu, C. and Neufeld, A.: Compactness criterion for semimartingale laws and semimartingale optimal transport. \emph{Trans.
		Amer. Math. Soc.}, {\bf 372, No. 1} (2019), 197--231.
	
	\bibitem{Sch71} 
	Schaefer, H.H.: {\bf Topological Vector Spaces},  Springer, New York, 1971.
	
	\bibitem{Smir51}
	Smirnov, Y.M.: On the metrization of topological spaces.
	Usp. {M}at. {N}auk. \textbf{6, No 6 (46)}, (1951), 100--111.
	
	\bibitem{SmFo76}
	Smolyanov, O., and Fomin, S.V.: Measures on linear topological spaces. \emph{Russ. Math. Surveys} \textbf{31}, (1976), 1--53.
	
\end{thebibliography}


\end{document}